\documentclass[letterpaper, 10 pt, conference]{ieeeconf}
\IEEEoverridecommandlockouts                              

\usepackage[utf8]{inputenc}
\usepackage[tbtags]{amsmath}
\usepackage{bbm}
\usepackage[subnum]{cases}
\usepackage{amssymb, psfrag, epstopdf}

\usepackage{amsthm,balance}
\usepackage{enumerate, makecell}
\usepackage{graphicx, textcomp}
\usepackage{color,cite,hyperref,mathtools,booktabs}
\usepackage{cleveref}
\usepackage{algorithm,algpseudocode}

\allowdisplaybreaks

\newtheorem{thm}{Theorem}
\newtheorem{lemma}{Lemma}
\newtheorem{rmk}{Remark}
\newtheorem{assumption}{Assumption}

\DeclareMathOperator{\diag}{dg}
\DeclareMathOperator{\rank}{rank}
\DeclareMathOperator{\trace}{Tr}

\DeclareMathOperator{\subjectto}{s.to}
\newtheorem{cor}{Corollary}

\newcommand \mcC{\mathcal{C}}
\newcommand \mcE{\mathcal{E}}
\newcommand \mcS{\mathcal{S}}
\newcommand \mcG{\mathcal{G}}
\newcommand \hmcG{\hat{\mathcal{G}}}
\newcommand \uX{\underline{X}}
\newcommand \oX{\overline{X}}
\newcommand \mcV{\mathcal{V}}
\newcommand \hmcE{\hat{\mathcal{E}}}

\newcommand \mcZ{\mathcal{Z}}
\newcommand \tL{\tilde{L}}
\newcommand \tW{\tilde{W}}

\newcommand \mcA{\mathcal{A}}
\newcommand \mcB{\mathcal{B}}

\newcommand \mcN{\mathcal{N}}

\allowdisplaybreaks

\def\BibTeX{{\rm B\kern-.05em{\sc i\kern-.025em b}\kern-.08em
    T\kern-.1667em\lower.7ex\hbox{E}\kern-.125emX}}
 \title{Optimal Design of Power Grid Topologies for Improved Stability}	
 
\begin{document}
\title{Efficient Topology Design Algorithms for Power Grid Stability}

\author{Siddharth Bhela, Harsha Nagarajan, Deepjyoti Deka, and Vassilis Kekatos
}
	
	\maketitle
	\thispagestyle{empty}
	\pagestyle{empty}

\begin{abstract}
		The dynamic response of power grids to small disturbances influences their overall stability. This paper examines the effect of network topology on the linearized time-invariant dynamics of electric power systems. The proposed framework utilizes ${\cal H}_2$-norm based stability metrics to study the optimal placement of lines on existing networks as well as the topology design of new networks. The design task is first posed as an NP-hard mixed-integer nonlinear program (MINLP) that is exactly reformulated as a mixed-integer linear program (MILP) using McCormick linearization. To improve computation time, graph-theoretic properties are exploited to derive valid inequalities (cuts) and tighten bounds on the continuous optimization variables. Moreover, a cutting plane generation procedure is put forth that is able to interject the MILP solver and augment additional constraints to the problem on-the-fly. The efficacy of our approach in designing optimal grid topologies is demonstrated through numerical tests on the IEEE 39-bus network.
\end{abstract}

	\section{Introduction} \label{sec:introduction}
	Widespread adoption of new grid technologies is continuously changing the face of our modern electricity networks. Increased penetration of renewable energy sources and changing load patterns has lead to higher volatility in power networks~\cite{gayme}. The stochastic nature of renewables and active loads is likely to produce recurring disturbances that will require careful planning and design of power networks with stability in mind \cite{gayme}. Additionally, the loss of rotational inertia in systems with high percentage of renewables will significantly reduce the capability of power grids to handle such disturbances \cite{ulbig2014impact}.
	\allowdisplaybreaks
	
	While several works have explored the placement of virtual inertia to improve the dynamic performance of power networks \cite{Poolla17}, it is not well understood how grid topology affects transient stability \cite{GuoLow18}. 
	Recent works have shown that the grid Laplacian matrix eigenvalues allows us to quantify the effect of grid topology on the power network \cite{GuoLow18}. Our work is further motivated by the claim that grid robustness against low frequency disturbances is determined by the connectivity of the network \cite{GuoLow18}. Past studies have looked at designing grid topologies for specific goals such as reduction of transient line losses \cite{gayme}, improvement in feedback control \cite{mallada}, and coherence-based network design \cite{SDPvoltage}. Topologies can also be designed to pursue other objectives such as improving reliability, minimizing investment and operating costs, reducing line losses, and managing network congestion; see \cite{MM19}, \cite{KH2011} and references therein.
	
	
	This paper investigates the effect of topology on power system dynamics. For a variety of $\mathcal{H}_2$-norm based stability metrics, such as reduction of line losses, fast damping of oscillations, and network synchronization, our previous works \cite{Deka18ACC}, \cite{BDHK18} presented methods that can be used to optimally design power grid topologies. In \cite{BDHK18}, we presented reformulations of the combinatorial topology design task that allowed us to solve the problem to optimality, albeit with forbiddingly slow run-times. This work significantly improves the computational performance of the previous formulation using the following key contributions discussed in Section~\ref{sec:bounds}: \emph{i)} present methods to derive tight bounds on the continuous optimization variables; \emph{ii)} generate \emph{a-priori} valid cuts based on graph-theoretic properties; and \emph{iii)} showcase an eigenvector-based cut-generation procedure that is able to interject the solver and add constraints on-the-fly. To explore conditions under which a greedy scheme would perform well, we also present new conditions for supermodularity in Section~\ref{sec:subm}.
	
	
	
	\textbf{Notation:} Column vectors (matrices) are denoted by lower (upper) case letters, and sets by calligraphic symbols. The cardinality of set $\mathcal{X}$ is denoted by $|\mathcal{X}|$ and $\emptyset$ is an empty set. Given a real-valued sequence $\{x_1, \ldots, x_N\}$, $x$ is the $N \times 1$ vector obtained by stacking the entries $x_i$ and $\diag(\{x_i\})$ is the corresponding diagonal matrix. The operator $(\cdot)^{\top}$ stands for transposition. The $N \times N$ identity matrix is represented by $I_N$. The canonical vector $e_i$ has a $1$ at the $i$-th entry and is $0$ elsewhere. The time derivative of $\theta$ is denoted by $\dot{\theta}:=\frac{\delta \theta}{\delta t}$. $M \succeq 0$ implies that $M$ is positive semi-definite, and $X_{\mathcal{AB}}$ denotes the matrix obtained from $X$ upon sampling the rows and columns indexed respectively by sets $\mathcal{A}$ and $\mathcal{B}$.
	
	\section{Grid Modeling and Stability Metrics} \label{sec:formulation}

	An electric power network with $N+1$ nodes can be represented by a graph ${\cal G }= ({\cal V}, \cal{E})$, where ${\cal V}:= \{1, 2, \ldots,N+1\}$ corresponds to nodes and edges in ${\mcE}$ correspond to undirected lines. Let node $i=1$ be the reference, and collect the remaining nodes in set $\mathcal{V}_r:=\mathcal{V}\backslash\{1\}$.  The susceptance of line $(i,j)\in \mcE$ connecting nodes $i$ and $j\in \cal V$ is denoted by $b_{ij}>0$, while its reactance is denoted by $x_{ij}:=b_{ij}^{-1}$ under a lossless line model. Then, $L$ represents the $(N+1) \times (N+1)$ susceptance Laplacian matrix of graph $\cal G$ and is defined as 
	$L_{ij}:= 
\sum_{(i,j)\in {\mcE}}b_{ij},~\text{if $i=j$}; L_{ij}:=-b_{ij},~\text{if $(i,j) \in {\mcE}$};~\text{and}~0 ~\text{otherwise}$.
    
    We consider a small-signal disturbance setup where 
	each node $i \in \mathcal{V}$ is associated with a synchronous machine's rotor angle $\theta_i$, frequency $\omega_i= \dot{\theta_i}$, inertia constant $M_i $, and damping coefficient~$D_i$; see~\cite{kundur}. If a node $i$ hosts an ensemble of devices, the parameters $M_i$ and $D_i$ represent lumped characterizations of the collective behavior of the hosted devices~\cite{Poolla17}. Without loss of generality, the quantities $(\theta_i,\omega_i)$ will henceforth refer to deviations of nodal phase angles and frequencies from their steady-state values. Using these definitions, the state-space representation of the linearized power grid dynamics can be expressed as~\cite{kundur}
	\begin{align} \label{swingmatrixfull}
	\begin{bmatrix} \dot{{\theta}} \\ \dot{{\omega}}  \end{bmatrix}   = \underbrace{\begin{bmatrix} 0 & {I}_N\\ -{M}^{-1}L  & -{M}^{-1}{D} \end{bmatrix}}_{A:=} \begin{bmatrix} {\theta} \\ {\omega}  \end{bmatrix} + \underbrace{\begin{bmatrix} 0\\ {M}^{-1}\end{bmatrix}}_{B:=} {u}
	\end{align}
	where ${M}:=\diag(\{M_i\})$ and ${D}:=\diag(\{D_i\})$ are the $(N+1) \times (N+1)$ matrices collecting  the inertia and damping constants across all nodes; the $(N+1)$-long vectors ${\theta}$, ${\omega}$, and ${u}$ stack respectively the phase angles, frequencies, and power disturbances at each node. The subsequent analysis relies on the ensuing assumption, which has been justified in several existing works \cite{BDHK18,gayme,john}.
	
	
	\begin{assumption}\label{as:A1}
		The constants $(M_i, D_i)$ are positive and $D_i = c$ for all $i\in \cal V$ (identical damping).
	\end{assumption}
	
	Given the state-space model in \eqref{swingmatrixfull}, our goal is to design power network topologies that minimize the expected steady-state value of a generalized control objective combining angle deviations and frequency excursions~\cite{Deka18ACC}, \cite{Poolla17}
	\begin{equation}\label{loss}
	f(t) := \sum_{i\in\mathcal{V}}\left[\sum_{j\neq i} w_{ij}(\theta_i(t)-\theta_j(t))^2 + s_{i}\omega_i^2(t)\right]
	\end{equation}
	for given non-negative weights $\{w_{ij}\}$ and $\{s_i\}$. The weights $w_{ij}$ induce a connected weighted graph $\mcG_w$ that may be different from $\cal G$. Let $W$ be the Laplacian matrix of graph $\mcG_w$ and define $S:=\diag(\{s_i\})$. Then, it is easy to see that $f(t) = \|y(t)\|_2^2$, where
	\begin{align} \label{output}
	y(t) :=  \underbrace{\begin{bmatrix}W^{1/2}  ~&0 \\0 ~&S^{1/2}\end{bmatrix}}_{C:=}  \begin{bmatrix}\theta(t)\\ \omega(t)\end{bmatrix}.
	\end{align}
	
	The importance of the generalized control objective in \eqref{loss} is that by varying~$(W,S)$, one can capture different stability metrics and study them under a unified framework~\cite{Deka18ACC}. Note that the network coherence metric considered in the numerical tests corresponds to the case of $W = I_{N+1}- \frac{1}{N+1}1_{N+1}1_{N+1}^{\top}$ and $S =0$~\cite{Deka18ACC}. 
	
	The expected steady-state value of $f(t)$ can be interpreted as the squared ${\cal H}_2$-norm of the linear time-invariant (LTI) system described by \eqref{swingmatrixfull} and \eqref{output}. This system will be compactly denoted by $H:=(A,B,C)$. Leveraging this link, the generalized control objective can be expressed as
	\begin{equation} \label{obsh2}
	\|H\|^2_{{\cal H}_2} = \trace(B^{\top}QB)
	\end{equation}
	where $Q \succeq 0$ is the observability Gramian matrix of the LTI system $H$, and can be computed as the solution to the Lyapunov equation~\cite[Ch.~5]{boyd1991linear} 
	\begin{equation}\label{observability}
	A^{\top}Q+QA = -C^{\top}C. 
	\end{equation}
	The ensuing sections select network topologies that minimize the stability metric of \eqref{obsh2}.
	
	\section{Optimal Design of Grid Topologies} \label{sec:top}

	Consider a connected graph $\hat{\mathcal{G}}=(\mcV, \hat{\mcE})$, where $\hat{\mathcal{E}}$ is the set of all candidate lines. The goal is to find a subset of lines ${\mcE} \subseteq \hat{\mathcal{E}}$, so that the resultant power network minimizes the stability objective in \eqref{obsh2}. Because adding lines can be costly and utilities have limited budgets, we further impose the constraint that $|\mathcal{E}|\leq K$ with $K\geq N$. By setting $K=N$, a radial topology can be enforced. The topology design task can be now posed as~\cite{BDHK18}
	\begin{subequations}\label{eq:P1}
		\begin{align}
		\arg \min_{\mathcal{E} \in \hat{\mathcal{E}}}~&~\trace(B^{\top}QB)\\
		\mathrm{s.to}~ &~|{\mcE}| \leq K\label{eq:P1:cardinality}\\
 		&~Q \text{~satisfies}~\eqref{observability},~{\mcE} \text{~is connected.}
		\end{align}
	\end{subequations}
	
	Under Assumption~\ref{as:A1}, the objective of \eqref{eq:P1} can be shown to be proportional to $\trace(\tW\tL^{-1})$, where $\tW$ and $\tL$ are the $N\times N$ matrices obtained after removing the first row and column from $W$ and $L$, respectively; see \cite[Lemma 1]{BDHK18}. 
	 Problem~\eqref{eq:P1} can be then rewritten as~\cite{BDHK18}
	\begin{align}\label{eq:P2}
	\arg\min_{{\mcE} \in \hmcE} ~&~\trace(\tW\tL^{-1})\\
	\mathrm{s.to}~ &~ |{\mcE}| \leq K, \quad \rank(L) = N \nonumber
	\end{align}
	where the rank constraint ensures that $\mcE$ is connected. 

	
	To express the optimization in \eqref{eq:P2} over $\hmcE$ in a more convenient form, let us associate each line $\ell\in\hmcE$ with a binary variable $z_{\ell}$, which is $z_{\ell}=1$ if line $\ell$ is selected (that is $\ell \in\mcE$); and $z_{\ell}=0$, otherwise. If we collect variables $\{z_{\ell}\}_{\ell \in\hmcE}$ in vector $z$, then $z$ must lie in the set $\mcZ:=\left\{z:z^{\top}{1}_{|\hat{\mathcal{E}}|} \leq K, \ z \in \{0,1\}^{|\hmcE|}\right\}.$ Based on the line selection vector $z$, the reduced susceptance Laplacian of $\hmcG$ can be expressed as
	\begin{equation}\label{eq:L2}
	\tL(z)=\sum_{(i,j)\in\hmcE} z_{ij}b_{ij} a_{ij} a_{ij}^\top.
	\end{equation}
	Here, each $a_{ij} \in \{0,\pm1\}^{1 \times N}$ is the row vector of the reduced branch-bus incidence matrix  corresponding to line $(i,j)\in\hmcE$ \cite{BDHK18}. 
	Given the explicit form of $\tL(z)$, it is not hard to see that the objective in \eqref{eq:P2} is a monotone function that is minimized when all lines in $\hmcE$ are selected. Utilizing \eqref{eq:L2}, problem \eqref{eq:P2} can be equivalently written as an MINLP~\cite{BDHK18}
	\begin{subequations} \label{eq:P4}
		\begin{align}
		(X^*,z^*)\in\arg\min_{X,z\in \mcZ} ~&~\trace(\tW X)\label{eq:P4:cost}\\
		\mathrm{s.to}~ &~\tL(z)X = I_N. \label{eq:P4:con}
		\end{align}
	\end{subequations}
	Constraint \eqref{eq:P4:con} enforces $X=\tL^{-1}(z)$, and thus, matrix $\tL(z)$ to be full-rank at optimality. Problem \eqref{eq:P4} is non-convex due to the binary nature of $z$ and the bilinear constraints in \eqref{eq:P4:con}. To handle the latter, we adopt McCormick linearization~\cite{mccormick1976computability}, which is briefly reviewed next; see also \cite{BDHK18} for details.
	
	Constraint~\eqref{eq:P4:con} involves bilinear terms $z_{\ell}X_{ij}$ for all $\ell \in\hmcE$ and $i,j\in\mcV_r$. For every term, introduce an auxiliary variable 
	\begin{equation} \label{bilinear}
	y_{\ell ij}=z_{\ell}X_{ij}
	\end{equation}
	and let the entries $X_{ij}$ lie within bounds $[\uX_{ij},\oX_{ij}]$. Since $z_{\ell}$ is binary, the following inequalities hold true
	\begin{subequations}\label{mccormick}
		\begin{align}
		& y_{\ell ij} \geq z_{\ell}\uX_{ij}, \quad y_{\ell ij} \geq {X}_{ij} + z_{\ell}\oX_{ij} -\oX_{ij}\label{mccormick:b}\\ 
		& y_{\ell ij} \leq z_{\ell}\oX_{ij},  \quad y_{\ell ij} \leq {X}_{ij} + z_{\ell}\uX_{ij} -\uX_{ij}.\label{mccormick:d}
		\end{align}
	\end{subequations}
	One can replace the bilinear terms in \eqref{eq:P4:con} by $y_{\ell ij}$'s, drop equation  \eqref{bilinear}, and enforce \eqref{mccormick} as additional constraints for all $\ell\in\hmcE$ and $i,j\in\mcV_r$ to get an MILP reformulation of \eqref{eq:P4}. It is not hard to show that this reformulation is \emph{exact}, i.e.,  $y_{\ell ij}=z_{\ell}X_{ij}$ because $z$ is binary \cite{mccormick1976computability}. 
	
	Through the aforementioned process, problem \eqref{eq:P4} is reformulated to an MILP over variables $\{X_{ij}\}$, $\{z_{\ell}\}$, and $\{y_{\ell ij}\}$, and can thus be handled by modern MILP solvers such as Gurobi or CPLEX. Nonetheless, MILPs with McCormick linearization can be forbiddingly complex to solve if the bounds $(\uX_{ij},\oX_{ij})$ on each $X_{ij}$ are arbitrarily wide.
	
	\section{Valid Inequalities and Bound Tightening}\label{sec:bounds}
	This section develops graph theoretic arguments to provide easy-to-compute non-trivial bounds on the entries of $X$ and derive valid cuts for two topology design tasks: \emph{i)} adding lines to an existing connected network; and \emph{ii)} designing a new network afresh.
	\subsection{Augmenting Existing Power Networks} \label{subsec:augment}
	Consider an existing network described by $\mcG_e=(\mathcal{V},\mathcal{E}_e)$. The goal here is to select additional lines from $\hmcE\setminus \mcE_e$ to improve stability. This is an instance of problem \eqref{eq:P4} with the entries of $z$ related to the lines in $\mcE_e$ being fixed to one. Based on \eqref{eq:L2}, the reduced Laplacian matrices of the existing network $\mcG_e$ and network $\hmcG$ with all lines connected are $\tL_e:=\sum_{(i,j)\in\mcE_e} b_{ij} a_{ij} a_{ij}^\top$ and $\tL_f:=\sum_{(i,j)\in\hat{\mcE}} b_{ij} a_{ij} a_{ij}^\top$, respectively.
	Under this setup and assuming $\mcG_e$ is connected, the entries of $X$ minimizing \eqref{eq:P4} can be bounded as follows.
	\begin{lemma}\label{le:bounds1}
		Given that $\tL_f^{-1} \preceq X \preceq \tL_e^{-1}$, the diagonal entries of  $X$ are bounded by
		\begin{equation}\label{eq:bounds1d}
		[\tL_f^{-1}]_{ii} \leq X_{ii}\leq [\tL_e^{-1}]_{ii},\quad \forall i\in\mcV_r, 
		\end{equation}
		and its off-diagonal entries are bounded by
		\begin{subequations}\label{eq:bounds1od}
			\begin{align*}
			X_{ij}&\leq [\tL_f^{-1}]_{ij} +\sqrt{([\tL_e^{-1}]_{jj}-[\tL_f^{-1}]_{jj})([\tL_e^{-1}]_{ii}-[\tL_f^{-1}]_{ii})} \\
			X_{ij}&\geq [\tL_e^{-1}]_{ij} -\sqrt{([\tL_e^{-1}]_{jj}-[\tL_f^{-1}]_{jj})([\tL_e^{-1}]_{ii}-[\tL_f^{-1}]_{ii})}
			\end{align*}
		\end{subequations}
		for all $i, j \in\mcV_r$ with $j\neq i$.
	\end{lemma}
	
	\begin{proof}
		Since $X\succeq\tL_f^{-1}$, it follows that $v^\top(X-\tL_f^{-1})v\geq 0$ for all $v$.
		Selecting $v=e_i-\delta e_j$ for some $\delta\geq 0$ yields
		\begin{equation}\label{eq:b1}
		X_{ii} +\delta^2 X_{jj}-2\delta X_{ij}\geq [\tL_f^{-1}]_{ii}+\delta^2[\tL_f^{-1}]_{jj}-2\delta [\tL_f^{-1}]_{ij}.
		\end{equation}
		Setting $\delta=0$ provides the LHS of \eqref{eq:bounds1d}. The RHS of \eqref{eq:bounds1d} can be obtained by exploiting $\tL_e^{-1}\succeq X$ likewise.
		
		For the off-diagonal entries of $X$, rearrange \eqref{eq:b1} if $\delta>0$ and substitute $X_{ii}$ and $X_{jj}$ with the respective upper bounds from \eqref{eq:bounds1d} to obtain
        \begin{align*}
		    X_{ij} &\leq [\tL_f^{-1}]_{ij} +\frac{[\tL_e^{-1}]_{ii}-[\tL_f^{-1}]_{ii}}{2\delta}+ 
		{\delta}\frac{[\tL_e^{-1}]_{jj}-[\tL_f^{-1}]_{jj}}{2}.
        \end{align*}
		The RHS of the upper bound on $X_{ij}$ can be minimized over $\delta>0$ to obtain $\delta^\star:=\sqrt{\frac{[\tL_e^{-1}]_{ii}-[\tL_f^{-1}]_{ii}}{[\tL_e^{-1}]_{jj}-[\tL_f^{-1}]_{jj}}}.$
		Plugging $\delta^\star$ back into the bounds completes the proof. The lower bounds on $X_{ij}$'s can be obtained similarly starting from $\tL_e^{-1}\succeq X$.
	\end{proof}
	
%
	To provide some alternate bounds on the entries of $X$ that may be tighter, let us consider the following inequality \cite{ellen2011}
	\begin{equation}\label{eq:cut1}
		X_{ii} + X_{jj} - 2X_{ij} \leq d_{ij} + \epsilon \quad \forall~i,j \in \mcV_r
	\end{equation}
	where the LHS of \eqref{eq:cut1} is defined as the effective resistance of the graph $R_{ij}$ ; scalar $d_{ij}$ is equal to the sum of reactance values along the shortest path between nodes $i$ and $j$ on $\mcG_e$; and $\epsilon$ is an arbitrarily small positive number. The length $d_{ij}$ of the shortest path between all pairs $(i,j) \in \mcV_r$ can be obtained by using the Floyd-Warshall algorithm \cite{floyd62}. Note that for the special case where there is a unique path between nodes $i$ and $j$, the bound in \eqref{eq:cut1} is tight, that is $R_{ij}=d_{ij}$. Moreover, the effective resistance $R_{ij}$ does not increase when edges are added \cite{ellen2011}. This simple fact can be exploited to obtain the following bounds.
	 \begin{lemma}\label{UBcor3}
	 	The off-diagonal entries of $X$ are lower bounded by $X_{ij}  \geq \frac{[\tL_f^{-1}]_{ii} + [\tL_f^{-1}]_{jj} - d_{ij} - \epsilon}{2}$ for all $i,j \in \mcV_r$ with $j\neq i$.
	 \end{lemma}
	 \begin{proof}
	 	The bound can be simply obtained by rearranging the terms in \eqref{eq:cut1} and substituting the lower bounds for the diagonal entries $(X_{ii}, X_{jj})$ from \eqref{eq:bounds1d}.
	 \end{proof}
 
	Between the two lower bounds on the off-diagonal entries of $X$, we only keep the tighter one. Note that the reduced Laplacian matrix $\tL_e$ of the existing network $\mcG_e$ is invertible only if $\mcG_e$ is connected. If $\mcG_e$ is not connected, one could obtain bounds on the entries of $X$ by imposing a meshed or radial structure on the sought topology as discussed next.
	
	\subsection{Designing New Power Networks} \label{subsec:newdesign}
	This section considers designing a network afresh. In \cite{BDHK18}, we derived bounds useful for the design of radial topologies. Here we develop a new approach to derive bounds on $X_{ij}$'s for meshed networks. Heed the lower bound in \eqref{eq:bounds1d} is also valid for the design of new networks. Since $X$ is an inverse M-matrix~\cite{johnson82}, its off-diagonal entries also satisfy $X_{ij} \geq 0$.
	
	Exploiting the structure of $\hmcG=(\mcV,\hmcE)$, we can provide additional information on the entries of $X$ to accelerate \eqref{eq:P4} for designing radial and meshed grids alike. For example, if $\hmcG$ is disconnected upon removing edge $\ell\in\hmcE$, then $\ell$ belongs to the sought network topology and $z_\ell^*=1$ before solving \eqref{eq:P4}. Such critical edges $\mcE_c \subset \hmcE$ can be identified using the algorithm presented in our previous work \cite{BDHK18} and the entries of $z$ corresponding to these edges can be safely set to $1$. This process not only reduces the binary search for $z^*$ in \eqref{eq:P4}, but also tightens the lower bounds on certain $X_{ij}$'s. 
	
			\begin{cor}\label{UBcor1}
			Suppose a critical edge $\ell=(i,j) \in \hmcE$ partitions the nodes of $\hmcG$ into two disjoint connected components: set $\mcV_\ell$ and its complement $\bar{\mcV}_\ell$. If $\mcV_\ell$ contains nodes $(i,1)$ and $\bar{\mcV}_\ell$ contains node $j$, then
				\begin{equation}\label{eq:cor1}
			\frac{[\tL_f^{-1}]_{ii} + [\tL_f^{-1}]_{jj} - \hat{d}_{ij} - \epsilon}{2} \leq X_{ij}
			\end{equation}
		where $\hat{d}_{ij}$ is the length of the shortest path (reactance of edge $\ell$) between nodes $i$ and $j$ on graph $\hmcG$. 
		Similarly, for the design of radial networks one can tighten the lower bounds as
		\begin{equation}\label{eq:cor2}
		\hat{d}_{i1} - \epsilon \leq X_{ij}
		\end{equation}
		\end{cor}
	\begin{proof}
		Because edge $\ell$ is the only connection between nodes $i$ and $j$ on graph $\hmcG$, this edge is also the shortest path between the two nodes and must belong to the sought network topology. The bound in \eqref{eq:cor1} can then be shown to be valid using the arguments in the proof of Lemma~\ref{UBcor3}. To obtain the bound in \eqref{eq:cor2}, we refer to arguments presented in~\cite[Lemma 4]{BDHK18}.
		\end{proof}

So far we have obtained lower bounds on the entries of $X$, that is $\underline{X}\leq X$ for a known matrix $\underline{X}$. Using these entry-wise lower bounds, valid upper bounds on the entries of the symmetric matrix $X$ can be found by solving $\frac{N(N+1)}{2}$ linear programs of the form
\begin{subequations}\label{eq:LPs}
	\begin{align}
	\max_{X}~&~\alpha^\top X\beta\\
	\subjectto~&~\trace(\tW X) \geq \trace(\tW X^r), \\
	~&~\trace(\tW X) \leq \trace(\tW X^f),  \label{eq:LPs:c1} \\
	~&~X_{ii} +  X_{jj} - 2 X_{ij}\geq [\tL_f^{-1}]_{ii}+ [\tL_f^{-1}]_{jj} - 2[\tL_f^{-1}]_{ij} \label{eq:LPs:c2} \\
		~&~X_{ii} + X_{jj} - 2X_{ij} \leq \hat{d}_{ij} + \epsilon \quad \forall (i,j) \in \mcE_c \label{eq:LPs:c3} \\
		~&~X_{ii}\geq X_{ij} \quad \forall~i,j \in\mcV_r, i \neq j \label{eq:LPs:c4} \\
		~&~\underline{X}\leq X \label{eq:LPs:c5}
	\end{align}
\end{subequations}
where the vectors $(\alpha,\beta)$ in the cost can be set as: \emph{1)} $\alpha=\beta=e_i$ to upper bound $X_{ii}$; and \emph{2)} $\alpha=e_i$ and $\beta=e_j$ to upper bound $X_{ij}$. Matrix $X^f$ is any feasible solution to \eqref{eq:P4} and $X^r$ is the solution to the \emph{relaxed} MILP formulation of \eqref{eq:P4}, constraint \eqref{eq:LPs:c2} can be shown to be valid by simply substituting $\delta=1$ in \eqref{eq:b1}, the cut in \eqref{eq:LPs:c3} is derived using \eqref{eq:cut1} and Corollary~\ref{UBcor1}, and the inequality in \eqref{eq:LPs:c4} is a property of inverse M-matrices \cite{johnson82}. The MILP formulation of \eqref{eq:P4} with newly derived bounds in Sections~\ref{subsec:augment} and \ref{subsec:newdesign} together with \eqref{eq:LPs:c4} is henceforth referred to as the \emph{Tightened} MILP. Note that \eqref{eq:LPs:c4} is an important constraint that was also included in our previous MILP formulations \cite{BDHK18}.

	Our previous work in \cite{BDHK18} relied on the assumption that there exists a node in $\mcV$ that is incident to exactly one edge in $\hat{\mathcal{E}}$. By fixing this node to be the reference node, we were able to obtain several graph-theoretic bounds that are valid for the design of radial networks. This limiting assumption can be waived for the design of meshed networks, i.e., valid bounds can be found on $X_{ij}$'s, even if the chosen reference node has more than one edge incident on it.
	 For the special case where removing all connections to the reference node results in more than two connected components, matrix $X$ becomes block diagonal in structure and each sub-network (meshed or radial) can be designed independently by solving a separate MILP. In that sense, the problem in \eqref{eq:P4} is parallelizable. 

\begin{rmk} \label{rmk:ElemN}
	When considering the problem in Section~\ref{subsec:augment}, the assumption was that one is only augmenting the edges of a pre-existing network $\mathcal{G}_e$. However, elimination or addition of nodes (and associated edges) may also occur when conventional generation plants are phased out and new ones are built. We can account for these changes in our proposed framework. For instance, elimination of a node $n \in {\mathcal{V}}$ and its incident edges would result in a modified network $\mathcal{G}_e'=(\mathcal{V}',\mathcal{E}_e')$, where $\mathcal{V}'=\mathcal{V}\backslash \{n\}$ and $\mathcal{E}_e' \subset \mathcal{E}_e$. If $\mathcal{G}_e'$ is connected, then valid bounds on $X_{ij}$'s can be derived from Lemma~\ref{le:bounds1} by removing the $n$-th row and column of the Laplacian matrices $\tilde{L}_e$ and $\tilde{L}_f$. If $\mathcal{G}_e'$ is not connected, then one can obtain bounds on the entries of $X$ by enforcing a radial or meshed network topology using the process outlined in this subsection. Addition of nodes (and associated candidate edges) can be handled similarly.
\end{rmk}

So far we have presented techniques for finding bounds on the entries of $X$ \emph{a-priori}, that is before solving the MILP. For the design of new network topologies the derived bounds on $X_{ij}$'s may still be relatively loose. To accelerate the convergence of the solver, we next explore how valid global cuts can be added to the problem on-the-fly.

\subsection{Eigenvector-based Cuts for New Network Design}\label{subsec:eig}
	To see how valid cuts can be generated dynamically during the solution process, consider the constraint
	\begin{equation} \label{eq:Y}
	Y:=
	\begin{bmatrix}
		X & I_N \\
		I_N & \tL(z)
	\end{bmatrix}\succeq 0
	\end{equation}
	which is valid for \eqref{eq:P4}. One could replace \eqref{eq:P4:con} with \eqref{eq:Y}, but that would require solving a much harder mixed-integer semi-definite program (MISDP). Instead, observe that the SDP constraint in \eqref{eq:Y} could strengthen the bounds in MILP solvers and reduce the size of the branch-and-bound tree. Since enforcing such constraints is non-trivial, one can exploit the alternative characterization of an SDP matrix by selecting $S$ vectors $v_s\in\mathbb{R}^{2N}$ and augment the MILP in \eqref{eq:P4} with the linear cuts
	\begin{equation} \label{eq:SDPcut}
	v_s^{\top}Yv_s \geq 0, \quad s=1,\ldots,S.
	\end{equation}
	Vectors $\{v_s\}_{s=1}^S$ could be random, e.g., independently drawn as $v_s \sim \mcN(0, I_N)$. However, adding such constraints may not necessarily tighten the MILP formulation. We explore how $v_s$'s can be chosen judiciously to yield more meaningful cuts.
	 
	 \emph{$2k$-sparse eigenvector cuts:}
	 Let $(z^r, X^r)$ be the solution to the MILP formulation of \eqref{eq:P4} obtained by relaxing $z \in \{0,1\}^{|\hmcE|}$ to the box constraints $z\in [0,1]^{|\hmcE|}$. Moreover,  let $Y^r$ be the matrix obtained by substituting  $(z^r, X^r)$ in $Y$. Heed that the \emph{relaxed} MILP formulation is not exact, i.e., $(z^r, X^r)$ may not satisfy \eqref{eq:P4:con}. In fact, the pair $(z^r, X^r)$ may not satisfy \eqref{eq:Y} either. A valid cut can be derived from the latter observation, as explained below.  
	 
	 Another way of enforcing \eqref{eq:Y} is to ensure all the eigenvalues of $Y$ are non-negative. This can be accomplished by assigning $v_s$ in \eqref{eq:SDPcut} to be the eigenvectors corresponding to the negative eigenvalues of $Y^r$.
	Notice that the cuts in \eqref{eq:SDPcut} do not have to be added only once at the beginning of the solution process (root node), but can also be added dynamically by interjecting the MILP solver at every branch-and-bound node when a new relaxed solution is found for \eqref{eq:P4}. However, since the eigenvectors $\{v_s\}$ are generally non-sparse, each one of the linear inequality constraints $v_s^{\top}Yv_s$ will couple almost all entries in $Y$, that is roughly $N^2$ variables. Such dense constraints can be detrimental to the solver's overall solution time. To alleviate this, we include cuts stemming from a sparse $v_s$. To this end, let us partition $v_s^{\top}:= [v_{s_1}^{\top}~v_{s_2}^{\top}]$ such that
	\begin{equation}\label{eq:split}
	v_s^{\top}Yv_s = v_{s_1}^{\top}X v_{s_1}+v_{s_2}^{\top}\tL v_{s_2}+ 2v_{s_1}^{\top}v_{s_2}.
	\end{equation}
	Since constraint \eqref{eq:Y} is equivalent to $X\succeq 0$ and $X\succeq \tL^{-1}(z)$; see~\cite[Sec.~A.5.5]{BoVa04}, the first two entries in the summand of \eqref{eq:split} are non-negative and only the last term 
 	\begin{equation}\label{eq:sum}
 	2v_{s_1}^{\top}v_{s_2}=2\sum_{n=1}^N v_{{s_1},n}v_{{s_2},n}
 	\end{equation}
	contributes to $v_s^{\top}Yv_s$ being negative. Keeping this in mind, the idea here is to identify the $k$ most negative entries out of the $N$ summands in \eqref{eq:sum}. For the related indices $n$, we maintain the entries of $(v_{s_1},v_{s_2})$, whereas for the remaining indices we set the corresponding entries to zero. Thus, the modified vectors $(v_{s_1},v_{s_2})$ are $k$-sparse and bear the same sparsity pattern. Vector $v_s$ is then $2k$-sparse.
 	
	While the $2k$-sparse cuts can effectively enforce constraint \eqref{eq:Y} in the MILP formulation, there is a trade-off between the number of such cuts and improvement in solution time. Recall that because the cuts in \eqref{eq:SDPcut} are added to the solution process on-the-fly, adding too many cuts can significantly slow down the solver. 
	To circumvent this, we suggest adding only those vectors $\{v_s\}$ that correspond to the negative eigenvalues of $Y^r$ falling below a pre-defined threshold $\gamma$ and/or limit the total number of constraints added.
	
Beyond tightening the bounds on the entries of $X$ and adding constraints of the form in \eqref{eq:SDPcut}, constraints \eqref{eq:LPs:c2}-\eqref{eq:LPs:c3} can also be appended to the MILP formulation to improve computational efficiency.

     So far our framework has captured the dynamics of a power system where all nodes host synchronous machines. 
     Practical power networks however have zero-injection nodes with no dynamic behavior that can be eliminated via Kron-reduction \cite{Kron2013}. To see how the Kron-reduced Laplacian $L_{red}(z)$ relates to the full Laplacian matrix $\tL(z)$ and its inverse $X$, consider again the power system graph $\hat{\mathcal{G}}$ and partition its nodes $\mathcal{V}$ into synchronous $\mathcal{S}$ and zero-injection nodes $\bar{\mathcal{S}}$. Then it follows from \eqref{eq:P4:con} that \begin{equation}\label{eq:Kron}
	X=	\begin{bmatrix}
		X_{\mcS \mcS} & X_{\mcS \bar{\mcS}} \\
		X_{\bar{\mcS} \mcS} & X_{\bar{\mcS} \bar{\mcS}}
	\end{bmatrix} = 
	\begin{bmatrix}
		\tL_{\mcS \mcS}(z) & \tL_{\mcS \bar{\mcS}}(z) \\
		\tL_{\bar{\mcS} \mcS}(z) & \tL_{\bar{\mcS} \bar{\mcS}}(z)
	\end{bmatrix}^{-1}
	\end{equation}
	Applying the matrix inversion lemma for block matrices~\cite[Sec.~A.5.5]{BoVa04}, the inverse of the Kron-reduced Laplacian matrix $X_{\mcS \mcS}=L^{-1}_{red}(z)=[\tL_{\mcS \mcS}(z)-\tL_{\mcS \bar{\mcS}}(z)\tL_{\bar{\mcS} \bar{\mcS}}(z)^{-1}\tL_{\bar{\mcS} {\mcS}}(z)]^{-1}$.
 Hence, power systems with zero-injection nodes can be simply handled by replacing the cost in \eqref{eq:P4:cost} with $\trace(\tW_{\mcS \mcS} X_{\mcS \mcS})$. The remaining formulation in \eqref{eq:P4} remains unaltered and optimization occurs on the complete $X$. Nodes with passive loads can also be eliminated via Kron-reduction to yield an identical form of $X_{\mcS \mcS}$ \cite{Kron2013}.

While the focus has been on techniques to solve the topology design problem exactly, we next explore conditions under which a greedy scheme would perform well.
\section{Supermodularity of Edge Augmentation}\label{sec:subm}
	For a given finite set $\mcS$, a function $f: 2^{\mcS} \rightarrow \mathbb{R}$ is said to be supermodular if for all subsets $\mcA \subseteq \mcB \subseteq \mcS$ and all $\mcC \in \mcS \backslash \mcB$, it holds that $f(\mcA \cup \mcC)-f(\mcA) \leq f(\mcB \cup \mcC)-f(\mcB)$ \cite{submodulargreedy}. In other words, the returns due to selection of $\mcC$ are non-diminishing where adding elements to the larger set $\mcB$ gives larger gains. Minimizing a supermodular decreasing function is NP-hard. However, it is known that a greedy scheme that iteratively minimizes the objective $f(A \cup \{s\})$ for $s \in \mcS \backslash \mcA$ is at least $1-1/e\simeq 63\%$ close to the optimal cost~\cite{submodulargreedy}.
	\begin{lemma}[\cite{submodulargreedy}]\label{submodularlem}
		Let $A^*$ and $A^{g}$  be the global and best greedy minimizer to the supermodular decreasing function $f$, respectively. Then $\frac{f(A^*)-f(A^{g})}{f(\emptyset)-f(A^*)} \leq 1/e$.
	\end{lemma}
	In general, the edge addition problem is not supermodular and hence strong theoretical guarantees are not permissible. The following theorem lays a restrictive condition under which the set function $f(\hmcE)=\text{Tr}(\tW\tL^{-1})$ is a supermodular decreasing function of edge addition.
	\begin{thm}\label{supermodular_edg}
		The function $\text{Tr}(\tW\tL^{-1})$ is decreasing and supermodular in the addition of susceptance weighted edges to set ${\mathcal E_e} \subset \hmcE$ if 
		\begin{align}
		\left([\tL^{-1}_e]_{ik} - [\tL^{-1}_e]_{il}\right) - \left([\tL^{-1}_e]_{jk} - [\tL^{-1}_e]_{jl}\right) &> 0 \nonumber\\
		\left([\tL^{-1}_e]_{mk} - [\tL^{-1}_e]_{ml}\right) - \left([\tL^{-1}_e]_{nk} - [\tL^{-1}_e]_{nl}\right) &> 0 \nonumber\\
		\left([\tL^{-1}_e]_{im} - [\tL^{-1}_e]_{in}\right) - \left([\tL^{-1}_e]_{jm} - [\tL^{-1}_e]_{jn}\right) &> 0 \nonumber
		\end{align}
		where $(k,l)$ is an edge in $\tW$  ($w_{kl} >0$) while $(i,j)$ and $(m,n)$ are edges added to ${\tL_e}$. 
	\end{thm}
		\begin{proof}
			Let $\mcS = \{(i,j),(m,n)\}$ such that  $\mcS \not\in {\mathcal{E}_e}$. Define the function $f({\mathcal{E}_e}) = \text{Tr}(\tW\tL^{-1}_e)$ and the positive semi-definite matrix $\Delta =\sum_{r \in \mcS} b_r{a}_ra_r^{\top}$. The Neumann series expansion of $f({\mathcal{E}_e} \cup \{{\mathcal{S}}\})$ is given by
			\begin{subequations}
				\begin{align}
				\text{Tr}(\tW(\tL_e+\Delta)^{-1}) = \text{Tr}(\tW\tL_e^{-1}) -\text{Tr}(\tW\tL_e^{-1}\Delta\tL_e^{-1})+\nonumber\\
				\text{Tr}(\tW\tL_e^{-1}\Delta\tL_e^{-1}\Delta\tL_e^{-1}) + \text{higher order terms}\nonumber
				\end{align}
			\end{subequations}
			Note that the first-order term is negative for all $b_r \geq 0$ and hence the function is decreasing. To prove that $f(\hmcE)=\text{Tr}(\tW\tL^{-1})$ is supermodular, we find conditions under which the function is strictly convex. For this, the second derivative of $f(\hmcE)$ with respect to $(b_{ij}, b_{mn})$ must be positive
			\begin{equation}
			\small
			\frac{\delta^2 f}{\delta b_{ij}\delta b_{mn}} > 0
			\Rightarrow (a^{\top}_{mn}\tL_e^{-1}\tW\tL_e^{-1}a_{ij})(a^{\top}_{mn}\tL_e^{-1}a_{ij}) > 0 \label{doublederiv}
			\end{equation}
			A sufficient condition for positivity is when both terms on the LHS of \eqref{doublederiv} are positive. Expanding them we have
			\begin{subequations}
				\begin{align}
				&a^{\top}_{mn}\tL_e^{-1}\tW\tL_e^{-1}a_{ij}\nonumber\\
				&= \sum_{w_{kl} > 0}w_{kl}\left([\tL^{-1}_e]_{ik} - [\tL^{-1}_e]_{il} - [\tL^{-1}_e]_{jk}+[\tL^{-1}_e]_{jl}\right)\nonumber\\&\left([\tL^{-1}_e]_{mk} - [\tL^{-1}_e]_{ml} - [\tL^{-1}_e]_{nk}+[\tL^{-1}_e]_{nl}\right) \text{~~and}\nonumber\\
				&a^{\top}_{mn}\tL_e^{-1}a_{ij} = [\tL^{-1}_e]_{im} - [\tL^{-1}_e]_{in} - [\tL^{-1}_e]_{jm} + [\tL^{-1}_e]_{jn}. \nonumber
				\end{align}
			\end{subequations}
			This leads to the conditions for supermodularity.
		\end{proof}
	To the best of our knowledge, this is the first time explicit conditions for supermodularity have been provided for the topology design problem in power grids. Since these conditions are restrictive and do not hold in general, a greedy scheme is unlikely to perform well. 
	
	\section{Numerical Tests} \label{sec:nx}
	All tests were carried out on a 2.3 GHz Intel Dual-Core i5 laptop with 16GB RAM. The MILP formulations were solved in Julia/JuMP v0.6~\cite{jump} using Gurobi v8.1.1. We first tested the performance of the \emph{Tightened} MILP formulation for augmenting existing networks. For this design task, the IEEE $39$-bus system was used as the pre-existing connected network~\cite{39bus}. Set $\hmcE$ consisted of edges in this base network and an additional $22$ randomly placed lines. From these lines, we solved \eqref{eq:P4} for $K =\{5,6,7,8\}$. To satisfy Assumption~\ref{as:A1}, we assumed $M_i=10^{-4}$ on all nodes that did not host generators, and $D_i=c=0.025$ for all $i \in \mcV$. Table~\ref{table_1} compares the computation time of the MILP formulation in \cite{BDHK18} with the \emph{Tightened} MILP formulation discussed in this paper. On average, the run-time required to find the optimal solution to the \emph{Tightened} MILP \textit{decreased} by $61\%$.
	

	We next considered the radial topology design problem with $\hmcE$ composed of all edges in the IEEE $39$-bus network. To generate the eigenvector-based cuts we utilized a threshold value of $\gamma=-0.95$ and set the sparsity parameter $k$ to one. Compared to the MILP formulation presented in \cite{BDHK18} that required $8,034$ sec. to reach the optimal solution, the \emph{Tightened} MILP formulation with \eqref{eq:LPs:c2}-\eqref{eq:LPs:c3} was solved in $7,361$ sec. Augmenting the previous formulation with the eigenvector-based cuts reduced the run-time further to $4620$ sec. The overall reduction in computation time was $40\%$ - a non-trivial improvement in comparison to state-of-the-art methods. Similarly, we considered the optimal design of a meshed network with $39$ edges, where $\hmcE$ was composed of all edges in the  $39$-bus network. The optimal solution in this case was found in $9$ hours. Longer run-times for the latter task can be attributed to looser bounds on the entries of $X$. 

	\section{Conclusions}
	\label{sec:conc}
	We have presented tight mathematical formulations and numerous valid cutting planes that can substantially speed up the computational time required to find an optimal topology design (radial or meshed) for enhanced power system stability. 
	\begin{table}[t]
		\centering
		\small
		\caption{Comparison of computation Times}
		\begin{tabular}{cccc}
			\toprule
			Budget($K$) & MILP in \cite{BDHK18} (sec.) & \emph{Tightened} MILP (sec.) \\
			\cmidrule{1-3}
			$5$ &   $362$  & $149$  \\
			$6$  &  $545$  &$208$   \\
			$7$ &  $837$  &$301$     \\
			$8$ &  $936$ &$385$   \\
			\bottomrule
		\end{tabular}
		\vspace{-0.5cm}
		\label{table_1}
	\end{table}
	\bibliographystyle{IEEEtran}
	\bibliography{myabrv,references}
	
\end{document}